\documentclass[dvipsnames,a4paper,twoside]{article}
\usepackage{amsthm,amsmath,amssymb,xcolor,soul}
\usepackage[bookmarks=true,bookmarksopen=true]{hyperref}

\theoremstyle{plain}
\newtheorem{theorem}{Theorem}[section]
\newtheorem{lemma}[theorem]{Lemma}
\newtheorem{corollary}[theorem]{Corollary}
\newtheorem{proposition}[theorem]{Proposition}

\theoremstyle{definition}
\newtheorem{definition}[theorem]{Definition}

\theoremstyle{remark}
\newtheorem{remark}[theorem]{Remark}

 \numberwithin{equation}{section}

\newcommand{\pd}{\hat p}
\newcommand{\qd}{\hat q}

\usepackage{xspace}

\newcommand{\Pb}{\mbox{\rm (P)}\xspace}

\newcommand{\uad}{U_{\rm ad}}

\newcommand{\proj}{\operatorname{Proj}}
\title{Critical cones for sufficient second order conditions in PDE constrained optimization\thanks{The authors were partially supported by Spanish Ministerio de Econom\'{\i}a y Competitividad under research project MTM2017-83185-P.}}

\author{Eduardo Casas\thanks{Departmento de Matem\'{a}tica Aplicada y Ciencias de la Computaci\'{o}n, E.T.S.I. Industriales y de Telecomunicaci\'on, Universidad de Cantabria, 39005 Santander, Spain, {\tt eduardo.casas@unican.es}.}
\and Mariano Mateos\thanks{Departamento de Matem\'{a}ticas, Campus de Gij\'on, Universidad de Oviedo, 33203, Gij\'on, Spain, {\tt mmateos@uniovi.es}.}
}

\date{\url{https://epubs.siam.org/doi/10.1137/19M1258244}}
\begin{document}
\maketitle
\begin{abstract}
In this paper, we analyze optimal control problems governed by semilinear parabolic equations. Box constraints for the controls are imposed and the cost functional involves the state and possibly a sparsity-promoting term, but not a Tikhonov regularization term. Unlike finite dimensional optimization or control problems involving Tikhonov regularization, second order sufficient optimality conditions for the control problems we deal with must be imposed in a cone larger than the one used to obtain necessary conditions. Different extensions of this cone have been proposed in the literature for different kinds of minima: strong or weak minimizers for optimal control problems. After a discussion on these extensions, we propose a new extended cone smaller than those considered until now. We prove that a second order condition based on this new cone is sufficient for a strong local minimum.
\end{abstract}
\begin{quote}
\end{quote}

\begin{quote}
\end{quote}

\section{Introduction}
\label{S1}
Let us consider a domain $\Omega\subset\mathbb R^n$, $n\leq 3$, with a Lipschitz boundary $\Gamma$. Given $T>0$ we denote $Q=\Omega\times(0,T)$ and $\Sigma = \Gamma\times(0,T)$. In this paper, we investigate  second order sufficient optimality conditions for the control problem
\[
\Pb \min_{u \in \uad}  J(u):=F(u)+\mu j(u),
\]
where $\mu\geq 0$. Additionally, for $\mu>0$, we will further suppose that $\alpha < 0 < \beta$,
\[
\uad =\{u \in L^\infty(Q): \alpha \le u(x,t) \le \beta\ \text{ for a.a. } (x,t) \in Q\}
\]
with $-\infty < \alpha < \beta < +\infty$,
\[
F(u) =  \int_Q L(x,t,y_u(x,t)) \, dx\, dt + \nu_\Omega \int_\Omega L_\Omega(x,y_u(x,T)) dx,
\]
$\nu_\Omega\in\{0,1\}$,
and $j:L^1(Q)\to\mathbb{R}$ is given by $j(u) = \|u\|_{L^1(Q)}$.

Above $y_u$ denotes the state associated to the control $u$ related by the following semilinear parabolic state equation
\begin{equation}\label{E1.1}
\left\{\begin{array}{rcll}
\displaystyle\frac{\partial y_u}{\partial t} + A y_u + f(x,t,y_u) &=& u&\mbox{ in }Q,\\
 y_u& =& 0 &\mbox{ on }\Sigma,\\
  y_u(0)& = &y_0&\mbox{ in }\Omega.
  \end{array}
  \right.
\end{equation}
Assumptions on the data $A$, $f$, $y_0$, $L$ and $L_\Omega$ are specified in Section \ref{S2}.

It is well known that if $\bar u$ is a local minimum then first order necessary optimality conditions can be written as
\[J'(\bar u;u-\bar u)\geq 0\ \forall u \in \uad\]
while second order necessary optimality conditions  read like
\[F''(\bar u) v^2\geq 0\ \forall u\in C_{\bar u}\]
where $C_{\bar u}$ is the cone
\[C_{\bar u} =\{v\in L^2(Q)\mbox{ satisfying the sign condition }\eqref{EQU1.2}\mbox{ and }J'(\bar u;v) = 0\},\]
\begin{equation}\label{EQU1.2}
  v(x,t)\left\{\begin{array}{cl}
                    \geq 0 & \mbox{ if }\bar u(x,t)=\alpha, \\
                    \leq 0 & \mbox{ if }\bar u(x,t)=\beta.
                  \end{array}\right.
\end{equation}
The reader is referred to \cite[Theorem 3.7]{CHW2012a} for the elliptic case or \cite[Theorem 3.1. Case I]{CHW2016} for the parabolic case. 

It is well known that in finite dimensional optimization the cone used to establish necessary second order necessary optimality conditions is the same as the one used for sufficient second order conditions. However this not the case in general for optimization problems in infinite dimension; see the example by Dunn \cite{Dunn98}. Despite this, if the Tikhonov term $\frac{\gamma}{2}\|u\|^2_{L^2(Q)}$ with $\gamma > 0$ is present in the cost functional of the control problem, we can take the same cone for both necessary and sufficient conditions; see e.g., \cite{Bonnans1998}, \cite{Casas-Troltzsch2012} or \cite{Casas-Troltzsch2016} for the case $\mu = 0$, or  \cite{CHW2012a}, \cite{CHW2016} or \cite{CRT2014} for $\mu > 0$.  Other works that consider second order sufficient conditions for problems with no Tikhonov regularization are \cite{CMR2019}, \cite{CWW2017}, \cite{CWW2018}, and \cite{Christof-Wachsmuth2018}. The results in these works cannot be applied to our problem due to the facts that we deal with a semilinear parabolic equation, our controls depend both on space and time and we do not have any assumption on the structure of the adjoint state.

In this paper, the Tikhonov term is not present. Then, an approach to deal with second order sufficient conditions, as suggested by Dunn \cite{Dunn98} or Maurer and Zowe \cite{Maurer-Zowe79} among others, consists of extending the cone of critical directions $C_{\bar u}$.
As far as we know, two ways to enlarge the cone have been proposed in the literature. In the context of abstract optimization problems, following Maurer and Zowe \cite{Maurer-Zowe79}, one could replace the condition $J'(\bar u;v)=0$ by $J'(\bar u;v)\leq \tau \|v\|_{L^2(Q)}$ for some small $\tau > 0 $.
In optimal control problems we can take advantage of the structure of the problem to define a slightly smaller cone by taking
\begin{equation}E_{\bar u}^\tau=\Big\{v\in L^2(Q)\mbox{ satisfying }\eqref{EQU1.2}\mbox{ and } J'(\bar u;v)\leq \tau \big( \|z_v\|_{L^2(Q)} +\nu_\Omega
\|z_v(\cdot,T)\|_{L^2(\Omega)}\big)\Big\},
\label{E1.2}
\end{equation}
where $z_v$ is the derivative of the control-to-state mapping in the direction $v$; see \eqref{E2.4} below.
A second alternative to extend $C_{\bar u}$ is based on the observation that for functions $v\in L^2(Q)$ satisfying the sign condition \eqref{EQU1.2} we have
\begin{align*}\mbox{for }\mu = 0:\ J'(\bar u;v) = 0\iff &v(x,t)=0\mbox{ if }|\bar\varphi(x,t)|>0\\
\mbox{for }\mu > 0:\ J'(\bar u;v) = 0\iff & v(x,t)
\left\{\begin{array}{cl}
                    \geq 0 & \mbox{ if }\bar\varphi(x,t) = -\mu\mbox{ and }\bar u(x,t)=0  \\
                    \leq 0 & \mbox{ if }\bar\varphi(x,t) = +\mu\mbox{ and }\bar u(x,t)=0  \\
                    =0 & \mbox{ if }\Big| |\bar\varphi(x,t)| -\mu \Big | > 0
                  \end{array}\right.
\end{align*}
where $\bar\varphi$ is the adjoint state associated with $\bar u$, defined in \eqref{E2.13} below; see \cite{Casas2012}, \cite{CRT2014}, \cite{Casas-Troltzsch2016}, \cite{CWW2017}, \cite{CWW2018}. Then a natural extension can be done specifying a smaller set of points where the functions $v$ should vanish: given $\tau >0$ we define the extended cone
\begin{align*}
\mbox{for }\mu = 0:\ D_{\bar u}^\tau = &\{v\in L^2(\Omega)\mbox{ satisfying }\eqref{EQU1.2}\mbox{ and }  v(x,t)=0\mbox{ if }|\bar\varphi(x,t)|>\tau\}\\
\mbox{for }\mu > 0:\ D_{\bar u}^\tau = &\Bigg\{v\in L^2(\Omega)\mbox{ satisfying }\eqref{EQU1.2}\mbox{ and } \\
& \hspace{1.2cm}v(x,t)
\left\{\begin{array}{cl}
                    \geq 0 & \mbox{ if }\bar\varphi(x,t) = -\mu\mbox{ and }\bar u(x,t)=0  \\
                    \leq 0 & \mbox{ if }\bar\varphi(x,t) = +\mu\mbox{ and }\bar u(x,t)=0  \\
                    =0 & \mbox{ if }\Big| |\bar\varphi(x,t)| -\mu \Big | > \tau
                  \end{array}\right.\Bigg\}.
\end{align*}

The following question immediately arises: is one of these two extensions better than the other? The answer seems to be difficult because they are not easy to compare. However we solve this issue by choosing
$D_{\bar u}^\tau\cap E_{\bar u}^\tau$. The main goal of this paper is to prove that a second order optimality condition based on this cone along with the first order optimality conditions imply the strong local optimality of $\bar u$.

The plan of the paper is as follows. In Section \ref{S2} we establish the assumptions on the functions defining \Pb, recall some regularity results on the state equation and the linearized state equation and establish the differentiability properties of the control-to-state mapping. We also state necessary optimality conditions. In Section \ref{Se.3} we prove our main result, namely Theorem \ref{T3.1}. In Section \ref{Se.4} we comment about extensions and limitations of our main result.

Before ending this introduction let us mention that the methods used in this paper cannot be applied to the case of control problems governed by the Navier-Stokes system. This is due to the fact that our approach requieres $L^\infty(Q)$ bounds for the states; see Theorem \ref{T2.1}. For quasilinear parabolic equations, it seems possible to obtain similar bounds using the results in \cite{Casas-Chrysafinos2018}. Also it seems reasonable that estimates analogous to that of \eqref{E2.7} or \eqref{E2.12} hold, but the extension is not immediate and is beyond the scope of this paper. We refer the reader interested in optimal control problems governed by these types of equations to \cite{Casas-Chrysafinos2012}, \cite{Casas-Chrysafinos2015}, \cite{Casas-Chrysafinos2018},  \cite{CasasDhamo2011}, \cite{CMR2007}, \cite{Casas-Troltzsch2009}, \cite{Troltzsch-Wachsmuth05} for the case where the Tikhonov term is present in the cost functional.

\section{Assumptions and preliminary results}
\label{S2}
On the partial differential equation \eqref{E1.1}, we make the following assumptions.
\begin{itemize}
\item[(A1)] \label{A1} $A$ denotes the elliptic operator
\[
Ay =-\sum_{i,j=1}^n \partial_{x_j}(a_{i,j}(x)\partial_{x_i} y) + \sum_{j = 1}^n b_j(x,t)\partial_{x_j} y,
\]
where $b_j\in L^\infty(Q)$, $a_{i,j}\in L^\infty(\Omega)$, and the uniform ellipticity condition
\begin{equation*}
\exists \lambda_A>0 : \lambda_A|\xi|^2\leq \sum_{i,j=1}^n a_{i,j}(x)\xi_i\xi_j\ \mbox{ for all }\xi\in\mathbb R^n\mbox{ and a.a. }x\in \Omega
\end{equation*}
holds.
\item[(A2)] \label{A2} We assume that $f:Q\times\mathbb R\to\mathbb R$ is a Carath\'eodory function of class $C^2$ with respect to the last variable satisfying the following properties:
\begin{align*}
&\exists  C_f \in\mathbb R  : \frac{\partial f}{\partial y}(x,t,y)\geq C_f\ \forall y\in\mathbb R,
\\
&f(\cdot,\cdot,0)\in L^{\pd}(0,T;L^{\qd}(\Omega)) \ \text{ for some } \pd,\qd\geq 2 \text{ with } \frac{1}{\pd} + \frac{d}{2\qd}<1, \notag\\
&\forall M>0\ \exists C_{f,M}>0 : \left|\frac{\partial^j f}{\partial y^j}(x,t,y)\right|\leq C_{f,M}\ \forall |y|\leq M\mbox{ and } j=1,2,\notag\\
&\begin{array}{l}\forall \rho>0 \text{ and } \forall M>0\ \exists \varepsilon>0 \text{ such that}\\
\displaystyle\left|\frac{\partial^2 f}{\partial y^2}(x,t,y_1)-\frac{\partial^2 f}{\partial y^2}(x,t,y_2)\right|<\rho\  \forall|y_1|,|y_2|\leq M \text{ with } \ |y_1-y_2|<\varepsilon,\end{array}\notag
\end{align*}
for almost all $(x,t) \in Q$.
\end{itemize}
Examples of functions $f$ satisfying the above assumptions are the polynomials of odd degree with positive leading coefficients or the exponential function $f(x,t,y)= g(x,t)\mathrm{exp}(y)$ with $g\in L^\infty(Q)$, $g(x,t)\geq 0$ for almost all $(x,t) \in Q$.
\begin{itemize}
\item[(A3)]\label{A3} For the initial datum we assume $y_0\in L^\infty(\Omega)$.
\end{itemize}
On the functions $L$ and $L_\Omega$ defining the differentiable part $F$ of the cost functional $J$, we assume:
\begin{itemize}
\item[(A4)] \label{A4} $L:Q\times\mathbb R\to\mathbb R$ is a Carath\'eodory function of class $C^2$ with respect to the last variable satisfying the following properties:
\begin{align*}
&\begin{array}{l}
L(\cdot,{\cdot,}0)\in L^1(Q) \text{ and } \forall M>0 \ \exists \Psi_M\in L^{\pd}(0,T;L^{\qd}(\Omega)) \text{ and } C_{Q,M}\\ \text{ such that}\\  \displaystyle\left|\frac{\partial L}{\partial y}(x,t,y)\right|\leq \Psi_{M}(x,t) \text{ and }
\left|\frac{\partial^2 L}{\partial y^2}(x,t,y)\right|\leq C_{Q,M}\ \forall |y|\leq M,\end{array}
\\
&\begin{array}{l}\forall  \rho>0 \text{ and } \forall M>0\ \exists \varepsilon>0 \text{ such that}\\
\displaystyle\left|\frac{\partial^2 L}{\partial y^2}(x,t,y_1)-\frac{\partial^2 L}{\partial y^2}(x,t,y_2)\right|<\rho\ \forall |y_1|,|y_2|\leq M \text{ with } |y_1-y_2|<\varepsilon,\end{array}\notag
\end{align*}
for almost all $(x,t) \in Q$.

\item[(A5)] \label{A5} $L_\Omega:\Omega\times\mathbb R\to\mathbb R$ is a Carath\'eodory function of class $C^2$ with respect to the last variable satisfying the following properties:
\begin{align*}
&\begin{array}{l}L_\Omega(\cdot,0)\in L^1(\Omega) \text{ and } \forall M>0 \ \exists C_{\Omega,M} \text{ such that}\\  \displaystyle\left|\frac{\partial^j L_\Omega}{\partial y^j}(x,y)\right|\leq C_{\Omega,M}\ \forall |y|\leq M \mbox{ and }j=1,2\end{array}
\\
&\begin{array}{l}\forall  \rho>0 \text{ and } \forall M>0\ \exists \varepsilon>0 \text{ such that}\\
\displaystyle\left|\frac{\partial^2 L_\Omega}{\partial y^2}(x,y_1)-\frac{\partial^2 L_\Omega}{\partial y^2}(x,y_2)\right|<\rho\ \forall |y_1|,|y_2|\leq M \text{ with } |y_1-y_2|<\varepsilon,\end{array}\notag
\end{align*}
for almost all $x \in \Omega$.
\end{itemize}
Let us comment that the classical tracking-type cost functional
\[F(u) = \frac12\int_Q (y_u(x,t)-y_d(x,t))^2\, dx\, dt + \frac{\nu_\Omega}{2}\int_\Omega (y_u(x,T)-y_\Omega(x))^2\,dx\]
satisfies the above assumptions if $y_d\in L^{\pd}(0,T;L^{\qd}(\Omega))$ and $y_\Omega\in L^\infty(\Omega)$.

Hereafter, these hypotheses will be assumed without further notice throughout the rest of the work.

\subsection{Analysis of the state equation}
In this section we analyze the existence, uniqueness and some regularity properties for the solution of \eqref{E1.1} as well as its dependence with respect to the control $u$. We also prove some technical results to be used in the proof of our main result, Theorem \ref{T3.1}.
\begin{theorem}\label{T2.1}
For every $u\in L^{\pd}(0,T;L^{\qd}(\Omega))$ there exists a unique solution of \eqref{E1.1}, $y_u\in L^2(0,T;H^1_0(\Omega))\cap L^\infty( Q)$. Moreover, there exist positive constants $K_{\pd,\qd}$, $C_{\pd,\qd}$ and $M_\infty$ such that for all $u,\bar u\in\uad$,
\begin{align*}
& \|y_u\|_{L^2(0,T;H^1_0(\Omega))}+\|y_u\|_{L^\infty(Q)}\leq\\ &\hspace{2cm} K_{\pd,\qd}(\|u\|_{L^{\pd}(0,T;L^{\qd}(\Omega))}+\|f(\cdot,\cdot, 0)\|_{L^{\pd}(0,T;L^{\qd}(\Omega))} + \|y_0\|_{L^\infty(\Omega)} ),\\
& \|y_u-y_{\bar u}\|_{L^\infty(Q)}\leq  C_{\pd,\qd}\|u-\bar u\|_{L^{\pd}(0,T;L^{\qd}(\Omega))},\\
& \|y_u\|_{L^\infty(Q)}\leq M_\infty.
\end{align*}
Finally, if $u_k\rightharpoonup u$ weakly in  $L^{\pd}(0,T;L^{\qd}(\Omega))$, then the strong convergence \[
\|y_{u_k}-y_u\|_{L^\infty(Q)}+\|y_{u_k}-y_u\|_{L^2(0,T;H^1_0(\Omega))}
+\|y_{u_k}(\cdot,T)-y_u(\cdot,T)\|_{L^\infty(\Omega)}\to 0\]
holds.
\end{theorem}

\begin{proof}To deal with the nonlinearity  in the state equation we can proceed as in \cite[Theorem 5.1]{Casas-1997}. Combining this approach with the well-known results for linear equations, see e.g.  \cite[Chapter III]{LSU}, existence, uniqueness, regularity and the first and third estimates follow easily.

To deduce the second estimate and the convergence properties, we introduce $w_k = y_{u_k}-y_u$. Subtracting the equations satisfied by $y_{u_k}$ and $y_u$ and using the mean value theorem  we get the existence of measurable functions $\hat y_k = y_u+\theta_k(y_{u_k}-y_u)$,  $0<\theta_k(x,t)<1$, such that
\[
\left\{\begin{array}{rcll}
\displaystyle\frac{\partial w_k}{\partial t} + A w_k + \displaystyle\frac{\partial f}{\partial y}(x,t,\hat y_k)w_k &=& u_k-u&\mbox{ in }Q,\\
 w_k& =& 0&\mbox{ on }\Sigma,\\
  w_k(0)& = &0&\mbox{ in }\Omega.
  \end{array}
  \right.
\]
From \cite[Theorem III-10.1]{LSU}, we deduce the existence of $C_{\pd,\qd}>0$ and $\gamma \in(0,1)$ such that
$\|w_k\|_{C^{\gamma,\gamma/2}(\bar Q)}\leq C_{\pd,\qd}\|u_k-u\|_{L^{\pd}(0,T;L^{\qd}(\Omega))}$. This proves the second estimate. Finally, since $C^{\gamma,\gamma/2}(\bar Q)$ is compactly embedded in $C(\bar Q)$ it is immediate to see that $\|w_k\|_{C(\bar Q)}\to 0$. In particular, $\|w_k(\cdot,T)\|_{L^\infty(\Omega)}\to 0$ holds.
Using this fact and multiplying the above equation by $w_k$ and making integration by parts we infer convergence $w_k\to 0$ in $L^2(0,T;H^1(\Omega))$.
\end{proof}

Hereafter, we denote $Y= L^2(0,T;H^1_0(\Omega))\cap L^\infty( Q)$ and  $G:L^{\pd}(0,T;L^{\qd}(\Omega)) \longrightarrow Y$ as the mapping associating to each control the corresponding state $G(u) = y_u$.

\begin{theorem}\label{T2.2}
The mapping $G$ is of class $C^2$. Moreover, for every $u,v,v_1,v_2\in L^{\pd}(0,T;L^{\qd}(\Omega))$,  we have that $z_v=G'(u)v$ is the solution of
\begin{equation}\label{E2.4}
\left\{\begin{array}{rcll}
\displaystyle\frac{\partial z}{\partial t} + A z + \displaystyle\frac{\partial f}{\partial y}(x,t,y_u)z &=& v&\mbox{ in }Q,\\
 z& =& 0&\mbox{ on }\Sigma,\\
  z(0)& = &0&\mbox{ in }\Omega,
  \end{array}
  \right.
\end{equation}
and $z_{v_1,v_2}=G''(u)(v_1,v_2)$ solves the equation
\begin{equation}\notag
\left\{\begin{array}{rcll}
\displaystyle\frac{\partial z}{\partial t} + A z + \displaystyle\frac{\partial f}{\partial y}(x,t,y_u)z &=& -\displaystyle\frac{\partial^2 f}{\partial y^2}(x,t,y_u)z_{v_1}z_{v_2} &\mbox{ in }Q,\\
 z& =& 0&\mbox{ on }\Sigma,\\
  z(0)& = &0&\mbox{ in }\Omega,
  \end{array}
  \right.
\end{equation}
where $z_{v_i} =G'(u)v_i$, $i=1,2$.
Moreover $z_v$ and $z_{v_1,v_2}$ are continuous functions in $\bar Q$.
\end{theorem}

For the proof the reader is referred, for instance, to \cite[Theorem 5.1]{Casas-Troltzsch2012}.

From the classical theory for linear parabolic partial differential equations, we know that for every $v\in L^2(Q)$ there exists a unique solution $z_v$ of \eqref{E2.4} in the space $C([0,T],L^2(\Omega)) \cap L^2(0,T;H^1_0(\Omega))$. Therefore the linear mapping $G'(u)$ can be extended to a continuous linear mapping $G'(u):L^2(Q)\to C([0,T],L^2(\Omega)) \cap L^2(0,T;H^1_0(\Omega))$.

The following estimates for $z_v$ will be used in the next sections.
\begin{lemma}\label{L2.3}Let $u\in \uad$ and $v\in L^2(Q)$ be arbitrary, and let $z_v = G'(u)v$ be the solution of \eqref{E2.4}. Then, there exist constants $C_{Q,2}$ and $C_{Q,1}$ independent of $u$ and $v$ such that
\begin{eqnarray}
\label{E2.5}\|z_v\|_{L^2(Q)}+ \|z_v(\cdot,T)\|_{L^2(\Omega)}&\leq& C_{Q,2} \|v\|_{L^2(Q)},\\
\label{E2.6}\|z_v\|_{L^1(Q)} + \|z_v(\cdot,T)\|_{L^1(\Omega)}  &\leq& C_{Q,1} \|v\|_{L^1(Q)}.
\end{eqnarray}
If, further, $v\in L^{\pd}(0,T;L^{\qd}(\Omega))$, then there exists a constant $C_{Q,\infty}$ independent of $u$ and $v$ such that
\begin{eqnarray}
\label{E2.7}\|z_v\|_{C(\bar Q)}&\leq& C_{Q,\infty} \|v\|_{L^{\pd}(0,T;L^{\qd}(\Omega))}.
\end{eqnarray}
\end{lemma}
\begin{proof}
First let us note that from Theorem \ref{T2.1} and our assumption on $f$ (A2) we have that
\begin{equation}\left|\frac{\partial^j f}{\partial y^j}(x,t,y_u(x,t))\right|\leq C_{f,M_\infty}\ \forall u\in U_{ad}\mbox{ and a.e. } (x,t)\in Q,\ j = 1,2.
\label{E2.8}
\end{equation}
Then \eqref{E2.5} and \eqref{E2.7} are classical; see for instance \cite[Chapter III]{LSU}.

The estimate \eqref{E2.6} for $\|z_v\|_{L^1(Q)}$ follows from \cite{Casas-Kunisch2016}; see also \cite{Boccardo1997,Casas-1997}.

To prove the estimate for $\|z_v(\cdot,T)\|_{L^1(\Omega)}$ we proceed as follows.
Consider the function $\psi_T=\mathrm{sign}(z_v(\cdot,T))\in L^\infty(\Omega)$ and let $\psi\in L^\infty(Q)\cap L^2(0,T;H_0^1(\Omega))$ be the unique solution of the problem
\[
\left\{\begin{array}{rcll}
-\displaystyle\frac{\partial \psi}{\partial t} + A^* \psi + \frac{\partial f}{\partial y}(x,t,y_u) \psi &=& 0&\mbox{ in }Q,\\
 \psi& =& 0&\mbox{ on }\Sigma,\\
  \psi(T)& = &\psi_T&\mbox{ in }\Omega,
  \end{array}
  \right.
\]
where $A^*$ is the adjoint of $A$ given by
\begin{equation}\label{E2.9}
A^*\psi =-\sum_{i,j=1}^n \partial_{x_j}(a_{j,i}(x)\partial_{x_i} \psi) - \sum_{j = 1}^n \partial_{x_j}( b_j(x,t) y).
\end{equation}
Multiplying the equation satisfied by $z_v$ by $\psi$ and integrating over $Q$, we obtain
\begin{equation}\label{E2.10}
  \int_Q \psi\left(\partial_t z_v + A z_v + \frac{\partial f}{\partial y}(x,t,y_u) z_v\right) dxdt = \int_Q v\psi dxdt.
\end{equation}
Integrating by parts in the first integral, we have
\begin{align*}
 & \int_Q \psi\big(\partial_t z_v + A z_v + \frac{\partial f}{\partial y}(x,t,y_u) z_v\big) dxdt
 =  \int_\Omega \left(\psi(x,T)z_v(x,T) -\psi(x,0)z_v(x,0)\right)dx \\
  &+ \int_Q z_v \left(-\partial_t \psi+ A^* \psi + \frac{\partial f}{\partial y}(x,t,y_u) \psi\right) dxdt \\
 = & \int_\Omega \psi_T(x) z_v(x,T)dx   =  \int_\Omega \mathrm{sign}(z_v(x,T))z_v(x,T) dx = \|z_v(\cdot,T)\|_{L^1(\Omega)}.
\end{align*}
Now using \eqref{E2.10}, we have that
\[\|z_v(\cdot,T)\|_{L^1(\Omega)}\leq \|\psi\|_{L^\infty(Q)}\|v\|_{L^1(Q)}.\]
Finally, it is enough to realize that  for some constant $C$ we have
\[\|\psi\|_{L^\infty(Q)}\leq C \|\psi_T\|_{L^\infty(\Omega)} = C\]
and the proof is complete.
\end{proof}

The following technical result will be used in the proof of Theorem \ref{T3.1}.
\begin{lemma}\label{L2.4}Consider $u,\,\bar u\in U_{ad}$ with associated states $y_u$ and $\bar y$, respectively. Set $z_{u-\bar u}=G'(\bar u)(u-\bar u)$ and consider the constants $C_{f,M_\infty}$ satisfying \eqref{E2.8} and $C_{Q,\infty}$ introduced in Lemma \ref{L2.3}. Then the following estimates hold:

\begin{align}
\text{If } \|y_{u}-&\bar y\|_{L^\infty(Q)} < \frac{2}{C_{f,M_\infty}C_{Q,\infty}|\Omega|^{1/\qd} T^{1/\pd}}  \text{ then}\label{E2.11}\\
&\|z_{u-\bar u}\|_{C(\bar Q)} < 2\|y_{u}-\bar y\|_{L^\infty(Q)}.\notag\\ \notag   \\
\text{If } \|y_{u}-&\bar y\|_{L^\infty(Q)} < \frac{1}{C_{f,M_\infty}C_{Q,\infty}}  \text{ then} \label{E2.12}\\
&\|z_{u- \bar u}\|_{L^2(Q)} +\nu_\Omega \|z_{u- \bar u}(\cdot,T)\|_{L^2(\Omega)}\notag\\
&\geq \frac{1}{2}\Big(\|y_{u}- \bar y\|_{L^2(Q)} + \nu_\Omega \|y_{u}(\cdot,T)-\bar y(\cdot,T)\|_{L^2(\Omega)}\Big).\notag
\end{align}
\end{lemma}
\begin{proof}
Define $\eta = y_u-(\bar y+ z_{u-\bar u})$. The function $\eta$ satisfies the equation
\[
\left\{\begin{array}{rcll}
 \displaystyle\frac{\partial \eta}{\partial t} +
 A\eta  + f(x,t,y_u) - f(x,t,\bar y) -\displaystyle\frac{\partial f}{\partial y}(x,t,\bar y) z_{u- \bar u} &=& 0&\mbox{ in }Q,\\
 \eta& =& 0&\mbox{ on }\Sigma,\\
 \eta(0) & = & 0&\mbox{ in }\Omega.
  \end{array}
  \right.
\]
Using a second order Taylor expansion, we have that there exists a measurable function $0<\theta(x,t)<1$ such that, if we name $\hat y=  \bar y+\theta (y_u-\bar y)$, we have that
\begin{equation}
\left\{\begin{array}{rcll}
 \displaystyle\frac{\partial \eta}{\partial t} +
 A\eta  + \displaystyle\frac{\partial f}{\partial y}(x,t,\bar y) \eta & = & -\displaystyle\frac{1}{2}\displaystyle\frac{\partial^2 f}{\partial y^2}(x,t,\hat y)(y_u-\bar y)^2&\mbox{ in }Q,\\
  \eta &= &0 &\mbox{ on }\Sigma,\\
  \eta(0) &=& 0&\mbox{ in }\Omega.
  \end{array}
  \right.
  \notag
\end{equation}

Let us prove the first estimate.
With the help of Assumption (A2), we deduce from \eqref{E2.7} and \eqref{E2.8} that
\begin{align}\|\eta\|_{C(\bar Q)}\leq \frac{1}{2}C_{f,M_\infty} C_{Q,\infty}|\Omega|^{1/\qd} T^{1/\pd} \|y_u-\bar y\|_{L^\infty(Q)}^2.
\notag
\end{align}
Using this and \eqref{E2.11}, we infer
\begin{align}
\|z_{u - \bar u}\|_{C(\bar Q)}  \le & \|\eta\|_{C(\bar Q)} + \|y_u-\bar y\|_{L^\infty(Q)} \notag\\
  \le  &\frac{1}{2}C_{f,M_\infty} C_{Q,\infty}|\Omega|^{1/\qd} T^{1/\pd} \|y_u-\bar y\|_{L^\infty(Q)}^2 + \|y_u-\bar y\|_{L^\infty(Q)} \notag \\ \leq & 2\|y_u-\bar y\|_{L^\infty(Q)}.\notag
\end{align}

For the second inequality, notice that using the uniform boundness of the admissible states, assumption (A2) and \eqref{E2.5}, we have that
\[\|\eta\|_{L^2(Q)}+\nu_\Omega \|\eta(\cdot,T)\|_{L^2(\Omega)}\leq \frac{1}{2} C_{Q,2}C_{f,M_\infty} \|y_u-\bar y\|_{L^\infty(Q)} \|y_u-\bar y\|_{L^2(Q)}.\]
Finally, using \eqref{E2.12}, we have that
\begin{align*}
\|y_u- &\bar y\|_{L^2(Q)}+ \nu_\Omega \|y_u(\cdot,T)-\bar y (\cdot,T)\|_{L^2(\Omega)}\notag\\
\leq &\|\eta\|_{L^2(Q)}+ \nu_\Omega  \|\eta(\cdot,T)\|_{L^2(\Omega)} +\|z_{u- \bar u}\|_{L^2(Q)}+ \nu_\Omega  \|z_{u-\bar u}(\cdot,T)\|_{L^2(\Omega)}\notag\\
 \leq &  \frac{1}{2} C_{Q,2}C_{f,M_\infty} \|y_u-\bar y\|_{L^\infty(Q)} \|y_u- \bar y\|_{L^2(Q)}+\|z_{u- \bar u}\|_{L^2(Q)} + \nu_\Omega \|z_{u-\bar u}(\cdot,T)\|_{L^2(\Omega)}\notag \\
 \leq &  \frac{1}{2}  \|y_u- \bar y\|_{L^2(Q)}+\|z_{u- \bar u}\|_{L^2(Q)} + \nu_\Omega \|z_{u-\bar u}(\cdot,T)\|_{L^2(\Omega)},\notag
\end{align*}
and the second inequality follows.
\end{proof}

\subsection{First and second order optimality conditions for \Pb}
We recall the definition of the cost functional $J(u) = F(u) + \mu j(u)$. Before establishing the optimality conditions satisfied by a local solution we address the differentiability of the functional $F$.

The next theorem follows from the chain rule, Theorem \ref{T2.2} and assumptions (A2) and (A3).
\begin{theorem}\label{T2.5}
The functional $F:L^{\pd}(0,T;L^{\qd}(\Omega)) \longrightarrow \mathbb{R}$ is of class $C^2$ and for every $u,v,v_1,v_2\in L^{\pd}(0,T;L^{\qd}(\Omega))$
\begin{align}
 F'(u)v &= \int_Q \varphi_u v \, dx, \notag\\
 F''(u)(v_1,v_2) &= \int_Q\left(\displaystyle\frac{\partial^2 L}{\partial y^2} (x,t,y_u) -\varphi_u\frac{\partial^2 f}{\partial y^2} (x,t,y_u)\right)z_{v_1}z_{v_2}\, dx\,dt  \notag \\
 & +\nu_\Omega\int_\Omega \displaystyle\frac{\partial^2 L_\Omega}{\partial y^2} (x,y_u(x,T)) z_{v_1}(x,T)z_{v_2}(x,T)\, dx.\notag
\end{align}
where $z_{v_i}=G'(u)v_i$, $i=1,2$ and   $\varphi_u\in Y$ is the adjoint state associated to $u$, i.e., it is the solution of
\begin{equation}\label{E2.13}
\left\{\begin{array}{rcll}
-\displaystyle\frac{\partial \varphi}{\partial t} + A^* \varphi + \displaystyle\frac{\partial f}{\partial y}(x,t,y_u)\varphi &=& \displaystyle\frac{\partial L}{\partial y}(x,t,y_u)&\mbox{ in }Q,\\
 \varphi& =& 0&\mbox{ on }\Sigma,\\
 \varphi(\cdot,T) &=& \nu_\Omega\displaystyle\frac{\partial L_\Omega}{\partial y}(x,y_u(\cdot, T))&\mbox{ in }\Omega,
  \end{array}
  \right.
\end{equation}
and $A^*$ denotes the adjoint operator of $A$ introduced in \eqref{E2.9}.

\end{theorem}

Assumptions (A1), (A4) and (A5) together with Theorem \ref{T2.1} imply, see \cite[Chapter III]{LSU}, that for every $u\in \uad$, $\varphi_u\in L^2(0,T;H^1_0(\Omega))\cap L^\infty(Q)$ and there exists a constant $K_\infty>0$ independent of $u$ such that
\begin{equation}
\|\varphi_u\|_{L^2(0,T;H^1_0(\Omega))} + \|\varphi_u\|_{L^\infty(Q)}\leq K_\infty\quad \forall u\in \uad.
\label{E2.14}
\end{equation}

\begin{remark}From the expressions for $F'(u)$ and $F''(u)$ established in the previous theorems it is immediate that they can be extended through the same formulas to continuous linear and bilinear forms, respectively, in $L^2(Q)$.
Moreover, assumptions (A2) and (A3), Theorem \ref{T2.1} and inequality \eqref{E2.14} imply the existence of some $M_2>0$ such that
\begin{equation}
\label{E2.15}|F''(u)(v_1,v_2)|\leq M_2\Big(\|z_{v_1}\|_{L^2(Q)}\|z_{v_2}\|_{L^2(Q)}+\nu_\Omega\|z_{v_1}(\cdot, T)\|_{L^2(\Omega)}\|z_{v_2}(\cdot, T)\|_{L^2(\Omega)}\big)
\end{equation}
for all $u\in \uad$ and $v_1,v_2\in L^2(Q)$, where $z_{v_i}=G'(u)v_i$, $i=1,2$.
\end{remark}

Finally, we notice that the directional derivative of $j$ at $u$ in the direction $v$ can be computed as
\begin{equation}\label{E2.16}
  j'( u;v) = \int_{u>0} v + \int_{u=0} |v| - \int_{u < 0} v.
\end{equation}
In what follows, we will write $J'(u;v)=F'(u)v+\mu j'(u;v)$. We will also denote $\partial j(u)$ as the subdifferential of $j$ at $u$ in the sense of convex analysis.


Existence of a global solution of \Pb follows in a standard way using Theorem 2.1; see e.g. \cite{Casas-Mateos2017}. Since \Pb is not a convex problem, we consider local solutions as well. Let us state precisely the different concepts of local solution.

\begin{definition}\label{De.2.7}
We say that $\bar u\in \uad$ is an $L^r(Q)$-weak local minimum of \Pb, with $r\in [1,+\infty]$, if there exists some $\varepsilon >0$ such that
\[
J(\bar u)\leq J(u)\quad\forall u\in \uad\mbox{ with }\|\bar u-u\|_{L^r(Q)}\leq \varepsilon.
\]
An element $\bar u\in\uad$ is said to be a strong local minimum of \Pb if there exists some $\varepsilon>0$ such that
\[
J(\bar u)\leq J(u)\quad \forall u\in \uad\mbox{ with }\|y_{\bar u}-y_u\|_{L^\infty(Q)}  \leq \varepsilon.
\]
We say that $\bar u\in \uad$ is a strict (weak or strong) local minimum if the above inequalities are strict for $u\neq \bar
u$.
\end{definition}

As far as we know, the notion of strong local solutions in the framework of control theory was introduced in \cite{BBS2014} for the first time; see also \cite{BS2016}.

\begin{lemma}
The following properties  hold:
\begin{enumerate}
\item $\bar u$ is an $L^1(Q)$-weak local minimum of \Pb if and only if it is an $L^r(Q)$-weak local minimum of \Pb for every $r\in (1,+\infty)$.
\item If $\bar u$ is an $L^r(Q)$-weak local minimum of \Pb for some $r < +\infty$, then it is an $L^\infty(Q)$-weak local minimum of \Pb.
\item If $\bar u$ is a strong local minimum of \Pb, then it is a $L^r(Q)$-weak local minimum of \Pb for all $r \in [1,\infty]$.
\end{enumerate}
\label{L2.1}
\end{lemma}
\begin{proof}Statement 1 is a consequence of the equivalence of all the $L^r(Q)$ topologies $(1\leq r<+\infty)$ in $\uad$.
Since $\|u\|_{L^r(Q)}\leq T^{1/r}|\Omega|^{1/r}\|u\|_{L^\infty(Q)}$, statement 2 follows. To prove statement 3 we use the second estimate in Theorem \ref{T2.1}:
\[\|y_u-\bar y\|_{L^\infty(Q)}\leq C_{\pd,\qd} \|u-\bar u\|_{L^{\pd}(0,T;L^{\qd}(\Omega))}\leq C_r \|u-\bar u\|_{L^r(Q)}
\]
 for all $r \geq \max\{\pd,\qd\}.$ Then statement 3 follows from statement 1 and the above inequality.
\end{proof}

Next we state first order optimality conditions.
\begin{theorem}\label{T2.9}
  Suppose $\bar u$ is a local solution of \Pb in any of the senses given in Definition \ref{De.2.7}. Then
  \begin{equation}\label{E2.17}
J'(\bar u; u-\bar u)\geq 0\ \forall u\in \uad
  \end{equation}
  holds.
Moreover,
   there exist $\bar y$ and $\bar\varphi$ in $Y$ and $\bar \lambda\in\partial j(\bar u)$ such that
\begin{subequations}
  \begin{equation}\label{E2.18a}
\left\{\begin{array}{rcll}
\displaystyle\frac{\partial \bar y}{\partial t} + A \bar y + f(x,t,\bar y) &=& \bar u&\mbox{ in }Q,\\
 \bar y& =& 0&\mbox{ on }\Sigma,\\
 \bar y(\cdot, 0) & = & y_0&\mbox{ in }\Omega,
  \end{array}
  \right.
\end{equation}
\begin{equation}\label{E2.18b}
\left\{\begin{array}{rcll}
-\displaystyle\frac{\partial\bar\varphi}{\partial t} + A^* \bar\varphi + \displaystyle\frac{\partial f}{\partial y}(x,t,\bar y)\bar\varphi &=& \displaystyle\frac{\partial L}{\partial y}(x,t,\bar y)& \mbox{ in }Q,\\
 \bar\varphi& =& 0&\mbox{ on }\Sigma,\\
 \bar\varphi(\cdot,T) & = & \nu_\Omega\displaystyle\frac{\partial L_\Omega}{\partial y}(x,\bar y(x,T))&\mbox{ in }\Omega,
  \end{array}
  \right.
\end{equation}
\begin{equation}\label{E2.18c}
\int_Q (\bar\varphi+\mu\bar\lambda) (u-\bar u)dx\,dt\geq 0\ \forall u\in\uad.
\end{equation}
  \end{subequations}

\end{theorem}
\begin{proof}
  To prove \eqref{E2.17} it is enough to use the local optimality of $\bar u$ and the convexity of $\uad$ as follows:
  \[0\leq \lim_{\rho\searrow 0} \frac{J(\bar u+\rho(u-\bar u))}{\rho} = J'(\bar u;u-\bar u)\quad \forall u\in\uad.\]
From the expression of $F'$ established in Theorem \ref{T2.5} and the convexity of $j$ we infer
\begin{align*}
0\leq & \lim_{\rho\searrow 0 } \frac{J(\bar u+\rho(u-\bar u))}{\rho}\\
 \leq  & \lim_{\rho\searrow 0}
\frac{F(\bar u+\rho(u-\bar u))}{\rho}+\mu j(u)-\mu j(\bar u) \\
 = & \int_Q \bar\varphi (u-\bar u)dx\,dt + \mu j(u)-\mu j(\bar u)\ \forall u\in\uad.
 \end{align*}
Hence, $\bar u$ solves the problem
\[\min_{u\in L^\infty(Q)} I(u) := \int_Q \bar\varphi u dx\,dt+\mu j(u)+ I_{\uad} (u),\]
where $I_{\uad}$ is the indicator function of the convex set $\uad$. Therefore, using the subdifferential calculus, see e.g. \cite[Chapter I, Proposition 5.6]{Ekeland-Temam}, we obtain $0\in\partial I(\bar u) = \bar\varphi + \mu \partial j(\bar u) + \partial I_{\uad}(\bar u)$, which implies \eqref{E2.18c} for some $\bar\lambda\in \partial j(\bar u)$.
\end{proof}

   From \eqref{E2.18c} we deduce the following corollary; see \cite{CHW2016}.
\begin{corollary}
Under the assumptions of Theorem \ref{T2.9},
\begin{align*}
  \mbox{ if }\bar\varphi(x,t) > +\mu & \mbox{ then }\bar u(x,t) = \alpha, \\
  \mbox{ if }\bar\varphi(x,t) < -\mu & \mbox{ then }\bar u(x,t) = \beta.
  \end{align*}
If $\mu > 0$, then
\begin{align*}
   &\mbox{ if }|\bar\varphi(x,t)| < \mu  \mbox{ then }\bar u(x,t) = 0, \\
  &\bar\lambda(x,t)  =  \proj_{[-1,+1]}\left(-\frac{1}{\mu}\bar\varphi(x,t)\right)
  \end{align*}
and $\bar\lambda\in Y$.
\label{C2.10}\end{corollary}

Let us write the second order necessary conditions. Given a control $\bar u\in \uad$ satisfying \eqref{E2.17}, we say that a function $v\in L^2(Q)$ satisfies the sign condition if
\begin{equation}\label{E2.19}
  v(x,t)\left\{\begin{array}{cl}
                    \geq 0 & \mbox{ if }\bar u(x,t)=\alpha, \\
                    \leq 0 & \mbox{ if }\bar u(x,t)=\beta.
                  \end{array}\right.
\end{equation}
Following \cite{CHW2012a,CHW2016}, we introduce the cone
\[C_{\bar u} = \{v\in L^2(Q)\mbox{ satisfying }\eqref{E2.19}\mbox{ and } J'(\bar u;v)=0\}.\]
We have the following proposition; see \cite[Lemma 3.5]{CHW2012a}.
\begin{proposition}
If $\bar u \in \uad$ satisfies \eqref{E2.17}, then
\begin{equation}\label{E2.20}
J'(\bar u;v)\geq 0\mbox{ for all }v \in L^2(Q) \mbox{ satisfying the sign condition \eqref{E2.19}}.
\end{equation}
As a consequence, it follows that $C_{\bar u}$ is a closed convex cone.
\label{P2.1}
\end{proposition}

If $\mu = 0$, we deduce from Corollary \ref{C2.10} that $\bar\varphi(x,t) v(x,t) = | \bar\varphi(x,t) v(x,t)|$ for every $v\in L^2(Q)$ satisfying  the sign condition  \eqref{E2.19}. Consequently the following identity holds.
\begin{equation}
C_{\bar u} =
\{v\in L^2(Q)\mbox{ satisfying }\eqref{E2.19}\mbox{ and } v(x,t)=0\mbox{ if }|\bar\varphi(x,t)|>0\}.
\label{E2.21}
\end{equation}
For $\mu > 0$, from Corollary \ref{C2.10} we also infer that
\begin{equation}
\begin{array}{rl}
C_{\bar u} =& \Bigg\{v\in L^2(Q)\mbox{ satisfying }\eqref{E2.19}\mbox{ and}\\
            &\hspace{1.2cm} \ v(x,t) \left\{\begin{array}{cl}
                    \geq 0 & \mbox{ if }\bar\varphi(x,t) = -\mu\mbox{ and }\bar u(x,t)=0  \\
                    \leq 0 & \mbox{ if }\bar\varphi(x,t) = +\mu\mbox{ and }\bar u(x,t)=0  \\
                    =0 & \mbox{ if }\Big| |\bar\varphi(x,t)| -\mu \Big | >0
                  \end{array}\right. \Bigg\};
\end{array}
\label{E2.22}
\end{equation}
see \cite{CRT2014} for a proof.

The second order necessary conditions are established in \cite[Theorem 3.7]{CHW2012a}. Although that result is stated for elliptic problems and a Tikhonov regularization term, the proof can be translated  to our setting with the straightforward changes.
\begin{theorem}\label{T2.11}
  Suppose $\bar u$ is a local solution of \Pb in any of the senses given in Definition \ref{De.2.7}. Then,
$F''(\bar u)v^2\geq 0\mbox{ for all }v\in C_{\bar u}$ holds.
\end{theorem}

\section{Second order sufficient conditions}\label{Se.3}
In this section, we establish the sufficient second order optimality conditions. In what follows, $\bar u$ will denote a control of $\uad$ satisfying \eqref{E2.17}. We denote by $\bar y$ and $\bar\varphi$ the associated state and adjoint state.

As mentioned in the introduction, we have to extend the cone $C_{\bar u}$ to formulate the second order sufficient conditions for optimality.

Looking at $J'(\bar u;v)$ for every $\tau > 0$ we consider the extended cone
\[G_{\bar u}^\tau=\Big\{v\in L^2(Q)\mbox{ satisfying }\eqref{E2.19}\mbox{ and } J'(\bar u;v)\leq \tau \big(\|z_v\|_{L^1(Q)}+\nu_\Omega\|z_v(\cdot,T)\|_{L^1(\Omega)}\big)\Big\}.\]
The extended cone $E_{\bar u}^\tau$ introduced in \eqref{E1.2} has been used in the literature to formulate the second order sufficient optimality conditions; see \cite{CRT2014}. The cone $G_{\bar u}^\tau$ introduced above is a smaller extension of $C_{\bar u}$ than $E_{\bar u}^\tau$. Indeed, given $E_{\bar u}^\tau$, for every
\[\tau'\leq \frac{\tau}{\sqrt{|\Omega|\max\{1,T\}}}\]
the embedding $G_{\bar u}^{\tau'}\subset E_{\bar u}^\tau$ holds.

On the other hand, using the characterizations of the cone $C_{\bar u}$ given by \eqref{E2.21} and \eqref{E2.22} the following extensions appear in a natural way as well.
\begin{align*}
 \mbox{If } \mu =0,\ D_{\bar u}^\tau = & \{v\in L^2(Q)\mbox{ satisfying }\eqref{E2.19}\mbox{ and } v(x,t)=0\mbox{ if }|\bar\varphi(x,t)|>\tau\}.\\
 \mbox{If }\mu >0,\
D_{\bar u}^\tau = &
 \Bigg\{v\in L^2(Q)\mbox{ satisfying }\eqref{E2.19}\mbox{ and}\\
                  &\hspace{1.2cm}  \ v(x,t) \left\{\begin{array}{cl}
                    \geq 0 & \mbox{ if }\bar\varphi(x,t) = -\mu\mbox{ and }\bar u(x,t)=0  \\
                    \leq 0 & \mbox{ if }\bar\varphi(x,t) = +\mu\mbox{ and }\bar u(x,t)=0  \\
                    =0 & \mbox{ if }\Big| |\bar\varphi(x,t)| -\mu \Big | >\tau
                  \end{array}\right. \Bigg\}.
\end{align*}
For the use of the cones $E_{\bar u}^\tau$ and $D_{\bar u}^\tau$ to formulate the second order sufficient optimality conditions and for a discussion of their application to the stability analysis of the control problem, the reader is referred to \cite{CRT2014}. In that paper it is proved that a sufficient second order condition based on the cone $D_{\bar u}^\tau$ leads to an $L^2(Q)$-weak local minimum, while the same condition based on the cone $E_{\bar u}^\tau$ implies that $\bar u$ is a strong local minimum. Hereafter we will prove that the condition based on the cone
\[C_{\bar u}^\tau = D_{\bar u}^\tau\cap G_{\bar u}^\tau\]
yields a strong local minimum $\bar u$.
Our main result is as follows:
\begin{theorem}\label{T3.1}
 Let $\bar u\in \uad$  satisfy the first order optimality condition \eqref{E2.17}. Suppose in addition that there exist $\delta>0$ and $\tau >0$ such that
 \begin{equation}
 F''(\bar u)v^2 \ge \delta\left(\|z_v\|^2_{L^2(Q)}+\nu_\Omega \|z_v(\cdot,T)\|^2_{L^2(\Omega)}\right)\quad \forall v \in C_{\bar u}^\tau, \label{E3.1}
\end{equation}
where $z_v = G'(\bar u)v$.
Then, there exist $\varepsilon>0$ and $\kappa>0$ such that
\begin{equation}\label{E3.2}
J(\bar u)+\frac{\kappa}{2}\left(\|y_u-\bar y\|^2_{L^2(Q)}+\nu_\Omega\|y_u(\cdot,T)-\bar y(\cdot,T)\|^2_{L^2(\Omega)}\right)  \leq J(u)
\end{equation}
for all $u\in \uad$ such that $\|y_u-\bar y\|_{L^\infty(Q)}<\varepsilon.$
\end{theorem}
Note that if $\tau < \tau'$, then $C^{\tau}_{\bar u}\subseteq C^{\tau'}_{\bar u}$, and hence without loss of generality we can suppose that, for $\mu>0$, $\tau<\mu$.
Throughout the proof of Theorem \ref{T3.1} we will use the following lemma. A proof of an analogous result can be found in \cite{CMR2019,Casas-Troltzsch2016}, so we omit it.
\begin{lemma}\label{Le.3.2}
For all $\rho>0$ there exists $\varepsilon_\rho>0$ such that for every $u\in \uad$ satisfying $\|y_u -\bar y\|_{L^\infty(Q)} <\varepsilon_\rho$, there holds
\begin{equation}\label{E3.3}
|\left[F''(\bar u+\theta(u-\bar u))-F''(\bar u)\right]v^2| \leq \rho \Big(\|z_v\|^2_{L^2(Q)}+\nu_\Omega \|z_v(\cdot,T)\|^2_{L^2(\Omega)}\Big)
\end{equation}
for all $v\in L^2(Q)$ and all $\theta \in [0,1]$, where $z_v = G'(\bar u)v$.
\end{lemma}

{\em Proof of Theorem \ref{T3.1}.}
Consider $u\in \uad$ such that
 $\|y_u -\bar y\|_{L^\infty(Q)} <\varepsilon$,
where $\varepsilon$  will be fixed later independently of $u$; see \eqref{E3.17} below.

A second order Taylor expansion yields the existence of $\theta\in(0,1)$ such that
\begin{align}
  F(u) =& F(\bar u)+F'(\bar u)(u-\bar u)+\frac{1}{2}F''(u_\theta)(u-\bar u)^2, \label{E3.4}
\end{align}
where $u_\theta = \bar u+\theta(u-\bar u)$. Using this and the convexity of  $j(\cdot)$, we have
\begin{align}
  J(u) = & F(u) + \mu j(u) \notag \\
  = & F(\bar u) + F'(\bar u)(u-\bar u) + \frac{1}{2}F''(u_\theta)(u-\bar u)^2 + \mu (j(u)-j(\bar u)) + \mu j(\bar u) \notag\\
\geq &  J(\bar u) + F'(\bar u)(u-\bar u) + \mu j'(\bar u;u-\bar u) + \frac{1}{2}F''(u_\theta)(u-\bar u)^2\notag\\
  =  &  J(\bar u) + J'(\bar u;u-\bar u) +\frac{1}{2}F''(\bar u)(u-\bar u) + \frac{1}{2}(F''(u_\theta)-F''(\bar u))(u-\bar u)^2. \label{E3.5}
\end{align}

In a first step, we will prove the existence of $\varepsilon_0$ such that
\begin{equation}\label{E3.6}
J(\bar u)+\frac{\delta}{4}\Big(\|z_{u-\bar u}\|^2_{L^2(Q)}+\nu_\Omega\|z_{u-\bar u}(\cdot,T)\|^2_{L^2(\Omega)}\Big) \leq J(u)\
\end{equation}
 for all $u\in \uad$ such that $\|y_u -\bar y\|_{L^\infty(Q)}  <\varepsilon_0$.
We will split the proof of this first step into three cases.

\medskip

{\em Case 1:} $u-\bar u\in C^\tau_{\bar u}$.
Applying Lemma \ref{Le.3.2} with $\rho = \delta/2$ we deduce the existence of $\varepsilon_1>0$ such that \eqref{E3.3} holds for every $u\in\uad$ such that $\|y_u-\bar y\|_{L^\infty(Q)} < \varepsilon_1$.
Inserting this  inequality in \eqref{E3.5} and using the variational inequality \eqref{E2.17} and the second order condition \eqref{E3.1}, we obtain
\begin{align}
  J(u)\geq  & J(\bar u) + \frac{\delta}{2} \Big(\|z_{u-\bar u}\|^2_{L^2(Q)}+\nu_\Omega \|z_{u-\bar u}(\cdot,T)\|^2_{L^2(\Omega)}\Big) \notag\\
        & \qquad -\frac{\delta}{4} \Big(\|z_{u-\bar u}\|^2_{L^2(Q)}+ \nu_\Omega \|z_{u-\bar u}(\cdot,T)\|^2_{L^2(\Omega)}\Big)\notag \\
  \geq  & J(\bar u) + \frac{\delta}{4} \Big(\|z_{u-\bar u}\|^2_{L^2(Q)}+ \nu_\Omega \|z_{u-\bar u}(\cdot,T)\|^2_{L^2(\Omega)}\Big).\notag
\end{align}

\medskip

{\em Case 2:} $u-\bar u\not\in G_{\bar u}^\tau$.  In this case, we consider
\begin{equation}
\notag
\varepsilon_2 = \min\left\{\varepsilon_1, \frac{2}{C_{f,M_\infty}C_{Q,\infty}T^{1/\pd}|\Omega|^{1/\qd}}, \frac{\tau}{\delta+M_2}\right\},
\end{equation}
where $\varepsilon_1$ is taken as in the previous case, and $C_{f,M_\infty}$, $C_{Q,\infty}$ and $M_2$ are introduced in \eqref{E2.8}, Lemma \ref{L2.3} and \eqref{E2.15}, respectively.
Then, from Lemma \ref{L2.4}, if $\|y_u-\bar y\|_{L^\infty(Q)} < \varepsilon_2$, we can estimate $\|z_{u-\bar u}\|_{C(\bar Q)}<2\varepsilon_2$.
Therefore we have
\begin{equation}
\|z_{u-\bar u}\|^2_{L^2(Q)} + \nu_\Omega \|z_{u-\bar u}(\cdot,T)\|^2_{L^2(\Omega)} \leq 2\varepsilon_2 \Big(\|z_{u-\bar u}\|_{L^1(Q)} + \nu_\Omega \|z_{u-\bar u}(\cdot,T)\|_{L^1(\Omega)}\Big).
\label{E3.7}
\end{equation}

Let us estimate the terms of \eqref{E3.5}. Since $u-\bar u$ satisfies the sign condition \eqref{E2.19} and $u-\bar u\not\in G_{\bar u}^\tau$, then with \eqref{E3.7} we get
\begin{align}
  J'(\bar u;u-\bar u) > &\tau \Big(\|z_{u-\bar u}\|_{L^1(Q)}+ \nu_\Omega \|z_{u-\bar u}(\cdot,T)\|_{L^1(\Omega)} \Big)\notag\\
  \geq  & \frac{\tau}{2\varepsilon_2} \Big(\|z_{u-\bar u}\|^2_{L^2(Q)} + \nu_\Omega \|z_{u-\bar u}(\cdot,T)\|^2_{L^2(\Omega)} \Big). \label{E3.8}
\end{align}
For the remaining terms, according to the choice we made for $\varepsilon_1$ in Case 1 and using \eqref{E2.15}, we infer
\begin{align}
    |F''(\bar u)(u-\bar u)^2|  +&  |[F''(u_\theta) - F''(\bar u)](u-\bar u)^2 |\notag \\
  \leq  & \left(M_2 +\frac{\delta}{2}\right) \Big(\|z_{u-\bar u}\|_{L^2(Q)}^2 + \nu_\Omega\|z_{u-\bar u}(\cdot,T)\|_{L^2(\Omega)}^2\Big).
  \label{E3.9}
\end{align}
From \eqref{E3.5}, \eqref{E3.8} and \eqref{E3.9} we deduce for $\|y_u-\bar y\|_{L^\infty(Q)}<\varepsilon_2$
\begin{align}
    J(u)
  \geq  & J(\bar u) + \left(\frac{\tau}{2\varepsilon_2}-\frac{M_2}{2}-\frac{\delta}{4}\right) \Big(\|z_{u-\bar u}\|^2_{L^2(Q)} + \nu_\Omega \|z_{u-\bar u}(\cdot,T)\|^2_{L^2(\Omega)} \Big) \notag\\
  \geq  & J(\bar u) + \frac{\delta}{4} \Big(\|z_{u-\bar u}\|^2_{L^2(Q)} + \nu_\Omega \|z_{u-\bar u}(\cdot,T)\|^2_{L^2(\Omega)} \Big).\notag
\end{align}

\medskip

{\em Case 3:} $u-\bar u\not\in D_{\bar u}^\tau$ and $u-\bar u\in G_{\bar u}^\tau$. Now we cannot use the second order condition \eqref{E3.1}, nor is the first derivative big enough to assure optimality. Hence, our method of proof is different from the previous two cases. First we define $\tau^* = \tau/\max\{1,C_{Q,1}\}\leq \tau$, where $C_{Q,1}$ is introduced in \eqref{E2.6}. If $u-\bar u\not\in G_{\bar u}^{\tau^*}$ holds, then we can argue as in the proof of the Case 2 to deduce that \eqref{E3.6} holds for $\|y_u - \bar y\|_{L^\infty(Q)} < \varepsilon_3$ with
\[
\varepsilon_3 = \min\left\{\varepsilon_2,\frac{\tau^*}{\delta+M_2}\right\}.
\]
Assume now that $u-\bar u\in G^{\tau^*}_{\bar u}$. Obviously $D_{\bar u}^{\tau^*}\subset D_{\bar u}^\tau$ holds, hence $u-\bar u\not\in D_{\bar u}^{\tau^*}$.

We define the set $W$ as follows:
\begin{align*}
\mbox{if }\mu =0,\quad W=\big\{ (x,t)\in  Q: &\
                    |\bar\varphi(x,t)|  >\tau\mbox{ and }u(x,t)-\bar u(x,t)\neq 0
                 \big\},
\\
\mbox{if }\mu >0,\quad W = \big \{ (x,t)\in  Q: &\
                       \bar\varphi(x,t) = -\mu\mbox{ and }\bar u(x,t)=0\mbox{ and }u(x,t) < 0,  \\
                 \mbox{or }&\      \bar\varphi(x,t) = +\mu\mbox{ and }\bar u(x,t)=0 \mbox{ and }u(x,t) >0,   \\
                  \mbox{or }&\  \Big| |\bar\varphi(x,t)| -\mu \Big | >\tau\mbox{ and }u(x,t)\neq \bar u(x,t)
                   \big\},
\end{align*}
and denote $V=Q\setminus W$.
Associated with $V$ we define the functions
\[
v(x,t)=\left\{\begin{array}{cc}
                    0 & \mbox{ if }(x,t)  \in W, \\
                    u(x,t)-\bar u(x,t) & \mbox{ if }(x,t)\in V
                                      \end{array}\right.
\]
and $w = (u - \bar u) - v$. We first notice three properties of $w$. In \cite[Proposition 3.6]{CRT2014} it is proved that
\begin{align}
J'(\bar u;w)\geq  \tau\|w\|_{L^1(W)} = \tau\|w\|_{L^1(Q)}.\label{E3.10}
\end{align}
Using this and the fact that the supports of $w$ and $v$ are disjoint, and noticing that $v$ satisfies the sign condition \eqref{E2.19}, which allows us to use \eqref{E2.20}, we obtain
  \begin{align}
J'(\bar u;u-\bar u) = J'(\bar u;v) + J'(\bar u;w) \geq J'(\bar u;v) + \tau\|w\|_{L^1(Q)} \ge \tau\|w\|_{L^1(Q)}.\label{E3.11}
  \end{align}
Finally, using \eqref{E2.6}, we have 
\begin{align}\|z_w\|_{L^1(Q)}+\|z_w(\cdot,T)\|_{L^1(\Omega)}\leq C_{Q,1} \|w\|_{L^1(Q)}\leq \max\{1,C_{Q,1}\} \|w\|_{L^1(Q)}.\label{E3.12}\end{align}
Regarding $v$, it is clear that $v\in D^\tau_{\bar u}$. From \eqref{E3.11} and \eqref{E3.12}   we get
\begin{align*}
  J'(\bar u;u-\bar u)  &   \ge  J'(\bar u;v) + \frac{\tau}{\max\{1,C_{Q,1}\}} \Big(\|z_w\|_{L^1(Q)}+ \nu_\Omega\|z_w(\cdot,T)\|_{L^1(\Omega)}\Big)\\
   &= J'(\bar u;v) + \tau^* \Big(\|z_w\|_{L^1(Q)}+\nu_\Omega \|z_w(\cdot,T)\|_{L^1(\Omega)}\Big).
\end{align*}
Since $u-\bar u\in G_{\bar u}^{\tau^*}$, we obtain
\begin{align*}
J'(\bar u;u- \bar u)\leq& \tau^*\Big(\|z_{u-\bar u}\|_{L^1(Q)}+ \|z_{u-\bar u}(\cdot,T)\|_{L^1(\Omega)}\Big) \\
\leq &\tau^*\Big(\|z_v\|_{L^1(Q)}+ \nu_\Omega\|z_v(\cdot,T)\|_{L^1(\Omega)}\Big) \\ &+ \tau^*\Big(\|z_w\|_{L^1(Q)}+ \nu_\Omega\|z_w(\cdot,T)\|_{L^1(\Omega)}\Big)
\end{align*}
Altogether, we conclude
\[
J'(\bar u;v) \leq \tau^*\Big(\|z_v\|_{L^1(Q)}+ \nu_\Omega\|z_v(\cdot,T)\|_{L^1(\Omega)}\Big).
\]
Therefore $v\in G_{\bar u}^{\tau^*}\subset G_{\bar u}^\tau$ and hence $v\in C^\tau_{\bar u}$ holds.

Now we combine the techniques of Cases 1 and 2. On one hand, we have that $v$ belongs to $ C^\tau_{\bar u}$, so that we can use the second order condition \eqref{E3.1}. On the other hand, the function $w$ satisfies that its $L^1(Q)$-norm bounds from below the directional derivative  $J'(\bar u;u-\bar u)$. Let us see in detail how to do this. We start at the inequality \eqref{E3.5}. Applying Lemma \ref{Le.3.2} we deduce the existence of $\varepsilon_4 > 0$ such that
\begin{equation}
|[F''(u_\theta) - F''(\bar u)](u - \bar u)^2| \le \frac{\delta}{4}\Big(\|z_{u-\bar u}\|^2_{L^2(Q)}+\nu_\Omega\|z_{u-\bar u}(\cdot,T)\|^2_{L^2(\Omega)}\Big)
\label{E3.13}
\end{equation}
for all $u\in U_{ad}$ such that $\|y_u - \bar y\|_{L^\infty(Q)} < \varepsilon_4$.
Now, we take
\begin{equation}
\varepsilon_0 = \min\left\{\varepsilon_3,\varepsilon_4,\frac{\tau^*}{M_2 + \frac{8M_2^2}{\delta}+\frac{21\delta}{4}}\right\}.
\notag
\end{equation}
From now on, we will assume that $\|y_u - \bar y\|_{L^\infty(Q)} < \varepsilon_0$. Using that $u - \bar u = v + w$ and applying the inequalities \eqref{E2.15}, \eqref{E3.1}, \eqref{E3.10} and \eqref{E3.13} we deduce from \eqref{E3.5}
\begin{align}
&J(u) \ge J(\bar u) + \tau\|w\|_{L^1(Q)} + \frac{1}{2}F''(\bar u)v^2 + \frac{1}{2}F''(\bar u)w^2\notag\\
& + F''(\bar u)(v,w) - \frac{1}{2}|[F''(u_\theta) - F''(\bar u)](u - \bar u)^2|\notag\\
&\ge J(\bar u) + \tau\|w\|_{L^1(Q)} + \frac{\delta}{2}\Big(\|z_v\|^2_{L^2(Q)}+\nu_\Omega \|z_v(\cdot,T)\|^2_{L^2(\Omega)}\Big)\notag\\
&-\frac{M_2}{2}\Big(\|z_w\|^2_{L^2(Q)}+\nu_\Omega \|z_w(\cdot,T)\|^2_{L^2(\Omega)}\Big)\notag\\
& - M_2\Big(\|z_v\|_{L^2(Q)}\|z_w\|_{L^2(Q)}+\nu_\Omega \|z_v(\cdot,T)\|_{L^2(\Omega)}\|z_w(\cdot,T)\|_{L^2(\Omega)}\Big)\notag\\
&- \frac{\delta}{8}\Big(\|z_{u - \bar u}\|^2_{L^2(Q)}+\nu_\Omega \|z_{u - \bar u}(\cdot,T)\|^2_{L^2(\Omega)}\Big).\label{E3.14}
\end{align}
Using the inequality $ab \le \frac{1}{2}a^2 + \frac{1}{2}b^2$ for appropriate real numbers $a, b$, we infer
\begin{align*}
&\|z_v\|_{L^2(Q)}\|z_w\|_{L^2(Q)}+\nu_\Omega \|z_v(\cdot,T)\|_{L^2(\Omega)}\|z_w(\cdot,T)\|_{L^2(\Omega)}\\
&\le \frac{\delta}{16M_2}\Big(\|z_v\|^2_{L^2(Q)}+\nu_\Omega \|z_v(\cdot,T)\|^2_{L^2(\Omega)}\Big)\\
&+\frac{4M_2}{\delta}\Big(\|z_w\|^2_{L^2(Q)}+\nu_\Omega \|z_w(\cdot,T)\|^2_{L^2(\Omega)}\Big).
\end{align*}
Inserting this estimate in \eqref{E3.14} and using \eqref{E3.12} and the definition of $\tau^*$, we obtain
\begin{align}
  J(u) \geq & J(\bar u) + \tau^*\Big(\|z_w\|_{L^1(Q)}+ \nu_\Omega\|z_w(\cdot,T)\|_{L^1(\Omega)}\Big)\notag  \\
   & +\frac{7\delta}{16}\Big(\|z_v\|^2_{L^2(Q)} + \nu_\Omega\|z_v(\cdot,T)\|^2_{L^2(\Omega)}\Big)\notag \\
   & - \left(\frac{M_2}{2}+4\frac{M_2^2}{\delta}\right)
   \Big(\|z_w\|^2_{L^2(Q)} + \nu_\Omega\|z_w(\cdot,T)\|^2_{L^2(\Omega)}\Big)\notag \\
   &  -\frac{\delta}{8}(\|z_{u-\bar u}\|^2_{L^2(Q)} + \nu_\Omega\|z_{u-\bar u}(\cdot,T)\|^2_{L^2(\Omega)}\Big).
   \label{E3.15}
\end{align}
Using that $u-\bar u=v+w$, we get
\begin{align*}
 \|z_v&\|^2_{L^2(Q)} + \nu_\Omega\|z_v(\cdot,T)\|^2_{L^2(\Omega)}  =
 \|z_{u-\bar u}-z_w\|^2_{L^2(Q)} + \nu_\Omega\|z_{u-\bar u}(\cdot,T)-z_w(\cdot,T)\|^2_{L^2(\Omega)}\\
  = & \Big(\|z_{u-\bar u}\|^2_{L^2(Q)} + \nu_\Omega\|z_{u-\bar u}(\cdot,T)\|^2_{L^2(\Omega)}\Big) + \Big(\|z_w\|^2_{L^2(Q)} + \nu_\Omega\|z_w(\cdot,T)\|^2_{L^2(\Omega)}\Big)\\
  & - 2 \Big(\|z_{u-\bar u}\|_{L^2(Q)} \|z_w\|_{L^2(Q)}+ \nu_\Omega\|z_{u-\bar u}(\cdot,T)\|_{L^2(\Omega)}\|z_w(\cdot,T)\|_{L^2(\Omega)}\Big)\\
  \geq &
  \frac{6}{7} \Big(\|z_{u-\bar u}\|^2_{L^2(Q)} + \nu_\Omega\|z_{u-\bar u}(\cdot,T)\|^2_{L^2(\Omega)}\Big)- 6 \Big(\|z_w\|^2_{L^2(Q)} + \nu_\Omega\|z_w(\cdot,T)\|^2_{L^2(\Omega)}\Big).
\end{align*}
Combining this with \eqref{E3.15}, we obtain
\begin{align}
  J(u) \geq & J(\bar u) + \tau^*\Big(\|z_w\|_{L^1(Q)}+ \nu_\Omega\|z_w(\cdot,T)\|_{L^1(\Omega)}\Big)\notag  \\
   & +\frac{\delta}{4}\Big(\|z_{u-\bar u}\|^2_{L^2(Q)} + \nu_\Omega\|z_{u-\bar u}(\cdot,T)\|^2_{L^2(\Omega)}\Big)\notag \\
   & - \left(\frac{M_2}{2}+4\frac{M_2^2}{\delta}+\frac{21\delta}{8}\right)
   \Big(\|z_w\|^2_{L^2(Q)} + \nu_\Omega\|z_w(\cdot,T)\|^2_{L^2(\Omega)}\Big).\label{E3.16}
\end{align}
Next we define the constants
\[C_{Q,3} =  2 C_{Q,\infty}^3(\beta-\alpha)^{2} (T+\nu_\Omega)|\Omega|\mbox{ and }\varepsilon_5 = \min\{ \varepsilon_2,8\frac{\varepsilon_0^3}{C_{Q,3}}\}, \]
where $C_{Q,\infty}$ is given in Lemma \ref{L2.3},
and assume $\|y_u-\bar y\|_{L^\infty(Q)}<\varepsilon_5$.
From \eqref{E3.11}, the fact that $u-\bar u\in G^{\tau}_{\bar u}$, Lemma \ref{L2.4}, and using that $\varepsilon_5\leq \varepsilon_2$, we deduce that
\begin{align*}
  \tau \|w\|_{L^1(Q)}\leq & J'(\bar u;u-\bar u)
  \leq
\tau\left(\|z_{u-\bar u}\|_{L^1(Q)} + \nu_\Omega\|z_{u-\bar u}(\cdot,T)\|_{L^1(Q)}\right)\\
\leq  &2\tau(|Q|+\nu_\Omega|\Omega|)\varepsilon_5= 2\tau(T+\nu_\Omega)|\Omega|\varepsilon_5.\end{align*}
Since $\|w\|_{L^\infty(Q)}\leq \beta-\alpha$, using the above inequality and  $\varepsilon_5^{1/3}\leq 2\varepsilon_0/C_{Q,3}^{1/3}$, we deduce
\begin{align*}
\|w&\|_{L^3(Q)} =  \left(\int_Q |w(x,t|^3 dx dt\right)^{1/3}
\leq  \left(\int_Q (\beta-\alpha)^2|w(x,t)| dx dt\right)^{1/3}\\
= & (\beta-\alpha)^{2/3} \left(\|w\|_{L^1(Q)}\right)^{1/3}
\leq (\beta-\alpha)^{2/3} \big(2(T+\nu_\Omega)|\Omega|)\big)^{1/3}\varepsilon_5^{1/3}
\leq \frac{2}{C_{Q,\infty}}\varepsilon_0.
\end{align*}
And using Lemma \ref{L2.3}, we obtain the estimate:
\[\|z_w\|_{L^\infty(Q)}\leq C_{Q,\infty}\|w\|_{L^3(Q)}\leq 2\varepsilon_0.\]
Using this, we have
\begin{align*}
  \tau^*\Big(\|z_w& \|_{L^1(Q)}  + \nu_\Omega\|z_w(\cdot,T)\|_{L^1(\Omega)}\Big)  \\ & - \left(\frac{M_2}{2}+4\frac{M_2^2}{\delta}+\frac{21\delta}{8}\right)
   \Big(\|z_w\|^2_{L^2(Q)} + \nu_\Omega\|z_w(\cdot,T)\|^2_{L^2(\Omega)}\Big) \\
   \geq & \left\{\frac{\tau^*}{2\varepsilon_0}-
   \left(\frac{M_2}{2}+4\frac{M_2^2}{\delta}+\frac{21\delta}{8}\right)\right\}
   \Big(\|z_w\|^2_{L^2(Q)} + \nu_\Omega\|z_w(\cdot,T)\|^2_{L^2(\Omega)}\Big) \geq 0,
\end{align*}
where the last inequality follows from the definition of $\varepsilon_0$. This combined with \eqref{E3.16} yields \eqref{E3.6}.

\bigskip

To conclude the proof, using the second part of Lemma \ref{L2.4}, with
\begin{equation}\label{E3.17}
\varepsilon = \min\left\{\varepsilon_0,\varepsilon_5,\frac{1}{C_{f,M_\infty}C_{Q,2}}\right\},
\end{equation}
and taking into account that $\nu_\Omega\in\{0,1\}$, we infer
\begin{align*}\|z_{u-\bar u}\|^2_{L^2(Q)}+\nu_\Omega&\|z_{u-\bar u}(\cdot,T)\|^2_{L^2(\Omega)}\\  \geq &\frac{1}{8}\Big(\|y_u-\bar y\|^2_{L^2(Q)}+\nu_\Omega\|y_u(\cdot,T)-\bar y(\cdot,T)\|^2_{L^2(\Omega)}\Big).
\end{align*}
Using this and \eqref{E3.6} we obtain
\[
  J(u) \geq J(\bar u)+\frac{\delta}{32}\Big(\|y_u-\bar y\|^2_{L^2(Q)}+\nu_\Omega\|y_u(\cdot,T)-\bar y(\cdot,T)\|^2_{L^2(\Omega)}\Big),
\]
and \eqref{E3.2} follows for $\kappa=\delta/16$.
\endproof

Notice that in Case 3 we did not use explicitly that $u-\bar u\not\in D^{\tau^*}_{\bar u}$. Observe that in case $u-\bar u\in D^{\tau^*}_{\bar u}$, then we would have that $w=0$ and $v=u-\bar u$, and Case 1 could be applied.
\section{Further extensions and limitations}
\label{Se.4}
The method developed in the previous sections can be extended with the obvious modifications to the case of a control problem governed by an elliptic equation as well as to Neumann control problems for both elliptic and parabolic equations.
However, let us mention two situations where it is difficult that the second order sufficient condition \eqref{E3.1} holds.

First, consider the situation where $L\equiv 0$ and $\nu_\Omega = 1$. In this case we have
\[F''(\bar u)v^2 = -\int_Q
 \bar \varphi \frac{\partial^2 f}{\partial y^2}(x,t,\bar y)z_{v}^2  \,dx\,dt+
\int_\Omega \frac{\partial^2 L_\Omega}{\partial y^2}(x,\bar y(x,T))z_v(x,T)^2\,dx.
\]
Looking at this expression it is easy to notice that the fulfillment of \eqref{E3.1} would depend on a lucky combination of the signs of the adjoint state and the second derivative of the nonlinearity $f$.
Consequently, Theorem \ref{T3.1} does not seem to be applicable to this problem.

A similar situation may occur if a nonlinearity is introduced on the boundary without a boundary observation. Consider, for instance, the problem governed by the elliptic equation
\[\min_{u\in \uad} F(u):=\frac{1}{2}\int_\Omega (y_u-y_d)^2 dx,\]
where $y_d\in L^2(\Omega)$ is given;
\[\uad = \{u\in L^\infty(\Omega):\ \alpha\leq u(x)\leq \beta\mbox{ for a.e. }x\in \Omega\},\]
with $-\infty<\alpha<\beta<\infty$; and
\[
\left\{\begin{array}{rcll}
-\Delta y_u &=& u&\mbox{ in }\Omega,\\
 \partial_{n} y_u+g(x,y_u(x))  & =& 0 &\mbox{ on }\Gamma.
  \end{array}
  \right.
\]
With the straightforward adaptations to this problem of the notation used along the paper, the second derivative of $F$ reads as
\[F''(\bar u)v^2 = \int_\Omega z_v^2 dx - \int_\Gamma \bar \varphi \frac{\partial^2 g}{\partial y^2}(x,\bar y)z_v^2 d\sigma(x).\]
In order to apply our theorem, the second order condition should be
\[F''(\bar u) v^2 \geq \delta \Big( \|z_v\|^2_{L^2(\Omega)} + \|z_v\|^2_{L^2(\Gamma)} \Big) \mbox{ for all }v\in C^\tau_{\bar u}.\]
Once again, this condition is unlikely to be fulfilled.

The situation would be different if we had a boundary observation $y_\Gamma\in L^\infty(\Gamma)$, so that the functional $F$ is given by
\[F(u) = \frac12\int_\Gamma (y_u(x)-y_\Gamma(x))^2 \, d\sigma(x).\]
Then we would get
\[F''(\bar u)v^2 =  \int_\Gamma \left(1-\bar \varphi \frac{\partial^2 g}{\partial y^2}(x,\bar y)\right)z_v^2 d\sigma(x)\]
and the second order sufficient condition
\[F''(\bar u) v^2 \geq \delta  \|z_v\|^2_{L^2(\Gamma)} \mbox{ for all }v\in C^\tau_{\bar u}\]
would have a chance to be fulfilled. For instance, if $\|\bar y -y_\Gamma\|_{L^2(\Gamma)}$ is small enough, then $\|\bar\varphi\|_{L^\infty(\Gamma)}$ is small as well, and, consequently we can deduce the existence of some $\delta >0$ such that $1-\bar \varphi \frac{\partial^2 g}{\partial y^2}(x,\bar y)\geq \delta$, which implies the above second order condition.

From the previous two cases we conclude that a nonlinearity in the whole domain requires a distributed observation and a boundary nonlinearity needs a boundary observation for fulfillment of the second order sufficient condition.

\end{document}